\documentclass[reqno]{amsart}
\usepackage[utf8]{inputenc}
\usepackage{xcolor,fullpage,hyperref,cleveref,enumitem}
\usepackage{tikz,amsmath,amssymb,color,mathtools}
\usetikzlibrary{calc}
\usepackage{float}
\usepackage{pgf}
\usepackage[backrefs]{amsrefs}
\usepackage{amsthm}
\usepackage{listings}

\theoremstyle{definition} 
\newtheorem{theorem}{Theorem}[section]
\newtheorem{corollary}{Corollary}[theorem]

\newtheorem{example}{Example}

\newtheorem{lemma}[theorem]{Lemma}
\newtheorem{conjecture}[theorem]{Conjecture}
\newtheorem{guess}[theorem]{Guess}
\theoremstyle{definition}

\theoremstyle{remark}

\newcommand{\NN}{\mathbb{N}}

\newcommand{\ZZ}{\mathbb{Z}}

\usepackage{thmtools} 
\usepackage{thm-restate}

\title{A symbolic computational approach to the generalized gambler's ruin problem in one and two dimensions}

\author[Martinez]{Lucy Martinez}
\address[L.~Martinez]{Department of Mathematics, Rutgers University, Piscataway, NJ 08854}
\email{\textcolor{blue}{\href{mailto:lucy.martinez@rutgers.edu}{lucy.martinez@rutgers.edu}}}

\begin{document}
\begin{abstract}
The power of symbolic computation, as opposed
to mere numerical computation, is illustrated with
efficient algorithms for studying the generalized gambler's ruin problem in one and two dimensions. We also consider a new generalization of the classical gambler's ruin where we add a third step which we call the mirror step. In this scenario, we provide closed formulas for the probability and expected duration.
\end{abstract}

\maketitle

\section{Introduction} 
Throughout we let $x$ be some positive integer such that $0< x < N$ where $N\in \NN=\{1,2,3,\ldots \}$. Consider a gambler who starts with $x$ dollars. At each gamble, the gambler either wins a dollar with probability $\frac{1}{2}$ or loses a dollar with probability $\frac{1}{2}$. The gambler's goal is to reach $N$ dollars without first running out of money (i.e., hitting $0$ dollars). If the gambler reaches $N$ dollars, we say that they are a \textit{winner}. The gambler continues to play until they either run out of money or win. This scenario is known as the \textit{gambler's ruin problem}, first posed by Pascal in 1656 in a letter to Fermat, as noted by Edwards \cite{Edwards1983}. In 1657, Christiaan Huygens restated the problem and published a solution for the probability of winning~\cite{Huygens}. For additional historical context, we refer the reader to a paper of Seongjoo Song and Jongwoo Song~\cite{Song}.

In this paper, we begin by providing an overview of the classical gambler's ruin problem, recalling results for both the probability of winning and the expected duration of the game. We also summarize analogous results on the generalized $1$-dimensional and $2$-dimensional versions of the gambler's ruin problem. Building on the $1$-dimensional version, we introduce a new generalization of the classical gambler's ruin game that includes an additional step called the \textit{mirror step}. For the generalized model, we derive formulas for the probability of winning and expected duration of the game. The objective of this paper is to propose an approach to reduce the computational running time required to determine the probability of winning and the expected duration specifically for the generalized 1-dimensional and 2-dimensional versions.

\subsection{Classical gambler's ruin problem}\label{subsec:classical}
Let $f(x)$ be the probability that the gambler exits the game as a winner starting with $x$ dollars. For $0<x<N$, this probability satisfies the recurrence relation
\begin{equation} \label{eq:probclass}
    f(x)=\frac{1}{2}f(x-1)+\frac{1}{2}f(x+1), \quad f(0)=0, f(N)=1.
\end{equation}
That is, if the gambler starts with $x$ dollars, then in the next round, the gambler has $x-1$ dollars or $x+1$ dollars, each with probability $\frac{1}{2}$. Using this recurrence relation and the boundary conditions, we can find the solution to be $f(x)=\frac{x}{N}$.

If the gambler starts with $x$ dollars, let $g(x)$ be the expected number of steps (expected duration of the game) the gambler takes to exit the game (either with $N$ dollars or $0$ dollars). Similar to the probability, for $1<x<N$, \[g(x)=\frac{1}{2}g(x-1)+\frac{1}{2}g(x+1) +1, \quad g(0)=0, g(N)=0.\] At each round, if the gambler has $x$ dollars, then in the next round, the gambler will have either $x-1$ dollars or $x+1$ dollars, each with probability $\frac{1}{2}$. However, we add 1 to the count since the gambler has taken one extra step. Using the recurrence relation and the boundary conditions, we can find the solution to be $g(x)=x(N-x)$.

Building upon the expected duration, we can obtain the probability generating function of the duration of the game. For a formal variable $t$ and $0<x<N$, \[F(x,t)=t\left(\frac{1}{2}F(x-1,t)+\frac{1}{2}F(x+1,t)\right), \quad F(0,t)=1, F(N,t)=1.\]
Taking the derivative of $F(x,t)$ with respect to $t$, and evaluating at $t=1$ recovers the expected duration of the game at $x$.

Consider extending the game so that the probability of losing one dollar or winning one dollar are not the same. In other words, let $p$ be the probability of winning one dollar, and $q=1-p$ be the probability of losing one dollar. Let $f(x)$ be the probability of exiting the game as a winner starting with $x$ dollars. Similarly to \Cref{eq:probclass}, we get
\[ f(x)=qf(x-1)+pf(x+1), \quad f(0)=0, f(N)=1, \text{ and } p+q=1. \]

Edwards gives a conjecture on how Pascal solved the above using a method of recursive formula \cite{Edwards1983}. We provide Edwards' solution to $f(x)$. Rewrite $f(x+1)-f(x)$ and observe the following

\begin{equation*}
    f(x+1)-f(x)=\frac{q}{p}\left(f(x)-f(x-1)\right)=\frac{q^2}{p^2}\left(f(x-1)-f(x-2)\right)=\cdots= \frac{q^i}{p^i}\left(f(1)-f(0)\right).
\end{equation*}

Hence, 

\begin{align*}
    f(x)&= \left ( \sum_{j=0}^{x-1}\left(\frac{q}{p}\right)^j \right ) f(1).
\end{align*}

Since $p=1-q$ then $p\neq q$, it follows that $\frac{q}{p}\neq 1$. Therefore, by the geometric series, we obtain

\begin{align*}
    1=f(N)&=\frac{1-\left(\frac{q}{p}\right)^N}{1-\frac{q}{p}} \cdot f(1).
\end{align*}

Thus, we can recover the following

\begin{align*}
 f(1)=\frac{1-\frac{q}{p}}{1-\left(\frac{q}{p}\right)^N}
\quad \text{and} \quad
 f(x)=\frac{1-\left(\frac{q}{p}\right)^x}{1-\left(\frac{q}{p}\right)^N}.
\end{align*}

\subsection{Generalized $1$-dimensional gambler's ruin}
The gambler's ruin problem can be formulated as follows: A particle starts at a point $x$ on a line of length $N$ where $0<x<N$. The particle moves to the left from $x$ to $x-1$ with probability $\frac{1}{2}$, or to the right from $x$ to $x+1$ with probability $\frac{1}{2}$.

Consider extending the 1-dimensional gambler's ruin game to include more than two steps on a line of length $N$. Let $r$ be a positive integer. Let $a_1,a_2,\ldots , a_r$ be distinct integers such that $a_1<a_2<\cdots<a_r$ where $a_1<0$ and $a_r>0$. Let $p_1, p_2,\ldots,p_r$ be probabilities such that $p_1+p_2+\ldots + p_r=1$, and let $P$  be the probability table $P=[[a_1,p_1], [a_2,p_2],\ldots , [a_r,p_r]]$ where each pair $[a_i,p_i]$ represents the outcome $a_i$ occurring with probability $p_i$. The generalized 1-dimensional gambler's ruin problem states that if a particle is currently at some $x$ then the particle moves from $x$ to $x+a_1$ with probability $p_1$, or moves from $x$ to $x+a_2$ with probability $p_2$, or moves from $x$ to $x+a_3$ with probability $p_3$, and so on.

Similarly to the method of solving a system of linear equations as in \Cref{subsec:classical}, we can obtain the probability of winning and the expected duration for the generalized case for any starting position. However, as $N$ grows, the computation time to solve a system of $N$ linear equations will be slower. In \Cref{sec:general 1d}, we present a faster method to reduce the computational running time by going from a system of $N-1$ linear equations with $N-1$ unknowns to a system of $a_r$ linear equations with $a_r$ unknowns where $a_r$ is the maximum of the steps in the probability table $P=[[a_1,p_1], [a_2,p_2],\ldots , [a_r,p_r]]$. Our method significantly drops the computational running time and we make comparisons between the direct approach and our strategy in \Cref{subsec:time comparison in 1d}.

\subsection{$2$-dimensional gambler's ruin}\label{subsec:2 dim}
Let $M$ and $N$ be positive integers. Consider a particle starting at a point $(x,y)$ in the interior of a rectangular grid of size $M\times N$, where $0<x<M$ and $0<y<N$. At each step, the particle moves in one of four directions, each with probability $\frac{1}{4}$: $(x-1,y), (x,y+1), (x+1,y),(x,y-1)$. The particle stops moving once it hits one of the four boundaries, defined by $x=0, x=M, y=0$, or $y=N$. 
By setting up a recurrence relation for the  expected duration, we can obtain a system of $(M-1)\times(N-1)$ linear equations with $(M-1)\times(N-1)$ unknowns. Andrej Kmet and Marko Petkov\v{s}ek gave an explicit formula involving a double sum, enabling direct computation of the expected duration for the $2$-dimensional gambler's ruin game without the need to solve systems of equations or use recursion \cite{Kmet}. While Kmet and Petkov\v{s}ek's formula expresses the expected duration as a double sum, our method reduces the computational running time by going from a system of $(M-1)\times(N-1)$ linear equations with $(M-1)\times(N-1)$ unknowns to a system of $N-1$ linear equations and $N-1$ unknowns. Our method is significantly faster than the direct approach and Kmet and Petkov\v{s}ek' formula. We make comparisons in \Cref{subsec:time comparison in 2d}.

One way to generalize the $2$-dimensional game is to change the probabilities of each of the four directions with probabilities $p_L, p_U, p_R,$ and $p_B$, corresponding to left, up, right and down movements, respectively where $p_L+p_U+p_R+p_B=1$. Although one can generalize the number of steps for either of the four directions, we focus on the case when the set of steps the particle can
move is $\{[0,1],[0,-1],[1,0],[-1,0]\}$,
 and remark that one can adapt our strategy for an {\it arbitrary} (finite) set of allowed steps, and arbitrary probability distribution.

\subsection{A mirror step variant of gambler's ruin}\label{subsec:mirror}
We consider a new generalization of the gambler's ruin problem. A particle starts at some point $x$ on a line of length $N$ where $0< x <N$. At each step, the particle either moves from $x$ to $x-1$ with probability $q_1$, or moves from $x$ to $x+1$ with probability $q_2$, or moves from $x$ to $N-x$ with probability $p$ where $0<p<1$ and $q_1+q_2+p=1$. We call this last step the \textit{mirror step}. The particle continues to walk on the line until it reaches $0$ or $N$. In this paper, we focus on the case when $q_1=q_2=\frac{1-p}{2}$ and we call this the \textit{symmetric case}.

We begin with an example where the particle starts at $x=1$ and generate data for different $p$ values with fixed $N$.

\begin{example}\label{ex: x=1 prob in intro}
Let $N=100$ and $x=1$. We generate data for the probability that if the particle is currently at $x=1$, the particle eventually ends at $100$. Let $p\in \{\frac{1}{2},\frac{1}{3},\frac{1}{4},\ldots, \frac{1}{10}\}$.

We use the procedure \texttt{Lk(p,x,N)}, as described in \Cref{appendix:mirror} in Maple which generates the following data in about $4.765$ seconds:
\begin{align*}
    T \coloneqq &[0.4142135624, 0.3660254038, 0.3333333333, 0.3090169944, 0.2898979486, 0.2742918852, 0.2612038750, \\
    &0.2500000000,0.2402530734].
\end{align*}

\noindent The sequence $T$ reads as follows. If the particle is currently at $x=1$ and $p=\frac{1}{2}$, the particle moves from $x$ to $x-1$ with probability $\frac{1-p}{2}=\frac{1}{4}$, or moves from $x$ to $x+1$ with probability $\frac{1-p}{2}=\frac{1}{4}$, or moves from $x$ to $100-x$ with probability $p=\frac{1}{2}$. Then, the probability of the particle starting at $x=1$ and ending at $100$ is $T_1=0.4142135624$. Similarly, if the particle is currently at $x=1$ and $p=\frac{1}{3}$, the particle moves from $x$ to $x-1$ with probability $\frac{1-p}{2}=\frac{1}{3}$, or moves from $x$ to $x+1$ with probability $\frac{1-p}{2}=\frac{1}{3}$, or moves from $x$ to $100-x$ with probability $p=\frac{1}{3}$. Then, the probability of the particle starting at $x=1$ and ending at $100$ is $T_2=0.3660254038$. Thus, $T_i$ is the probability of the particle starting at $x=1$ and ending at $100$ for $p=\frac{1}{i+1}$ where $1\leq i\leq 9$.

We then use the function 
\texttt{identify} in Maple on the sequence $T$. The function \texttt{identify} is based, in part, on the continued fraction expansion of any given numerical value. Using this function on the values of $T$, we conjecture that each of the probabilities in $T$ converges to
\begin{align*}
    M\coloneqq&\left[\sqrt{2}-1, \frac{\sqrt{3}-1}{2}, \frac{1}{3}, \frac{\sqrt{5}-1}{4}, \frac{\sqrt{6}-1}{5}, \frac{\sqrt{7}-1}{6},\frac{2 \sqrt{2}-1}{7}, \frac{1}{4}, \frac{\sqrt{10}-1}{9} \right].
\end{align*}
That is, the probability of the particle starting at $x=1$ and ending at $100$ converges to $M_i$ for $p=\frac{1}{i+1}$ where $1\leq i\leq 9$.
\end{example}

The previous example illustrates that when the particle starts at $x=1$, the probability of ending at $N$ converges fast. Denote this probability by $f_N^{(p)}(x)$. We state the following guess for $x=1$ and in \Cref{cor:probability at infinity} we provide a proof.

\begin{guess}\label{guess:x=1}
If the particle starts at $x=1$, then
\[\lim_{N\to\infty}f_N^{(p)}(1)=\frac{\sqrt{p}-p}{1-p}.\]
\end{guess}

In \Cref{sec:mirror step}, we provide other expressions for $\lim_{N\to \infty} f_N^{(p)}(x)$ when $x=2$ and $x=N-2$ with fixed $N$, and in \Cref{cor:probability at infinity} we provide the general formula for the limit.

This paper is structured as follows. In Section \ref{sec:general 1d}, we present a new approach to compute the probability of winning and the expected duration of the game that reduces the computational running time for the generalized $1$-dimensional gambler's ruin \footnote{All computations were performed using Maple on a laptop with an Intel Core i7-10510U processor (4 cores, 8 logical processors) and 8 GB of RAM.}. In Section \ref{sec:2d case}, we provide the analogous approach for the generalized $2$-dimensional case and compare the computational running times to a formula provided by Andrej Kmet and Marko Petkov\v{s}ek. In Section \ref{sec:mirror step}, we consider a mirror step variant of gambler's ruin and provide closed formulas for both the probability of winning and the expected duration of the game. We conclude with future directions in Section \ref{sec:future}.

\section{Generalized $1$-dimensional Gambler's Ruin}\label{sec:general 1d}

In this section, we introduce the recurrence relation for the probability of winning in the generalized $1$-dimensional gambler's ruin game, and introduce symbolic variables to the recurrence equation of the probability and expected duration. Recall that $P=[[a_1,p_1], [a_2,p_2],\ldots , [a_r,p_r]]$ denotes a probability table, where each pair $[a_i,p_i]$ represents the outcome $a_i$ occurring with probability $p_i$. To set up notation, we start with an example.

\begin{example}\label{ex:probability notation}
Let $N=5$ and $P$ be the probability table given by $P=\left[[-2,\frac{1}{2}],[1,\frac{1}{4}],[2,\frac{1}{4}]\right]$. If the particle starts at some $x$ where $0<x<5$ on the line of length $5$, then it can move along the line as follows: from $x$ to $x-2$ with probability $\frac{1}{2}$, or from $x$ to $x+1$ with probability $\frac{1}{4}$, or from $x$ to $x+2$ with probability $\frac{1}{4}$.
\end{example}

\subsection{Probability}\label{subsec:prob}
We now establish the recurrence relation for the probability that the particle reaches some position $\geq N$ starting from an initial position $x$. We then rewrite this recurrence and introduce new variables for the probabilities at each $x$.

Define $f(x)$ as the probability that a particle starting at $x$ will eventually reach a position $\geq N$. For $0<x<N$, this probability satisfies the recurrence relation
\begin{equation}\label{eq:generalprob}
    f(x)=\sum_{i=1}^r p_if(x+a_i),
\end{equation}
where $a_i\in \ZZ$, $a_1<a_2<\cdots <a_r$, $a_1<0$ and $a_r>0$. Unlike the classical gambler’s ruin problem, the generalized 1-dimensional scenario has more than two boundary conditions. Certainly $f(0)=0$ and $f(N)=1$. Since $a_1< a_2< \cdots< a_r$, the values $a_1$ and $a_r$ represent the minimum and maximum of all the integers $a_i$, respectively. For any integer $k$ such that $a_{1}+1 \leq k \leq 0$, it follows that $f(k)=0$. Indeed, if the particle is at $x=1$, it may move to $x+a_1=a_1+1$ with probability $p_1$. Given that $a_1<0$, this movement brings the particle to some position $k\leq 0$, implying that $f(k)=0$ for $k=a_1+1, a_1+2, \ldots, -1,0$. Similarly, for any integer $\ell$ such that $N \leq \ell \leq N+a_r-1$, we have $f(\ell)=1$. If the particle is at $x=N-1$, it may move to $x+a_r=N+a_r-1$ with probability $p_r$. Since $a_r>0$, this movement brings the particle to some position $\ell\geq N$, so $f(\ell)=0$ for $\ell=N,N+1,N+a_r-2,N+a_r-1$. 
Thus, there are $a_r-a_1$ boundary conditions.

\begin{example}\label{ex:slow version N=5}(Continuing \Cref{ex:probability notation})
Recall the probability table $P=\left[[-2,\frac{1}{2}],[1,\frac{1}{4}],[2,\frac{1}{4}]\right]$ and $N=5$. If $f(x)$ denotes the probability that the particle reaches some position $\geq 5$, then
\begin{equation*}
    f(x)=\frac{1}{2}f(x-2)+\frac{1}{4}f(x+1)+\frac{1}{4}f(x+2)
\end{equation*}
with initial and final conditions $f(-1)=f(0)=0$ and $f(5)=f(6)=1$. 

\noindent This setup results in a system of 4 linear equations for $0<x<5$,
\begin{align*}
   f(1)&=\frac{1}{2}f(-1)+\frac{1}{4}f(2)+\frac{1}{4}f(3)\\
   f(2)&=\frac{1}{2}f(0)+\frac{1}{4}f(3)+\frac{1}{4}f(4)\\
   f(3)&=\frac{1}{2}f(1)+\frac{1}{4}f(4)+\frac{1}{4}f(5)\\
   f(4)&=\frac{1}{2}f(2)+\frac{1}{4}f(5)+\frac{1}{4}f(6).
\end{align*}
Solving for the unknowns using the boundary conditions yields $f(1)= \frac{1}{5}, f(2)=\frac{13}{45}, f(3)=\frac{23}{45}, f(4)=\frac{29}{45}$.
\end{example}

Although this direct method works for small values of $N$, solving the system of $N-1$ equations becomes computationally expensive as $N$ grows. To address this, we rewrite the recurrence relation \Cref{eq:generalprob} as
\begin{equation}\label{eq:bettergeneralprob}
f(x)=\frac{1}{p_r}f(x-a_r)-\frac{1}{p_r}\sum_{i=1}^{r-1}p_if(x+a_i-a_r)
\end{equation}
obtained by the change of variables $x\mapsto x-a_r$. The boundary conditions remain, $f(a_1+1)= f(a_1+2)= \ldots= f(-1)=f(0)=0$ and $f(N)=f(N+1)=\ldots = f(N+a_r-2)= f(N+a_r-1)=1$. For each $1\leq j \leq a_r$ define $d_j=f(j)$ and construct the set $S=\{d_1,d_2,\ldots, d_j\}$ where $a_r$ is the maximum of the $a_i$'s. Using these variables, we express 
\begin{align}\label{eq:linear equations for better recurrence}
    f(a_r+1),f(a_r+2),\ldots, f(N+a_r-1)
\end{align}
as linear combinations of the elements in $S$.

Observe that $a_1<a_2<\cdots<a_r$ implies $a_i-a_r<0$ for any $1\leq i \leq r-1$. Hence, $x+a_i-a_r<x\leq N+a_r-1$ for all $x\in \{a_r+1, a_r+2,\ldots, N+a_r-1\}$ since $0<a_r+1<a_r+2< \cdots< N+a_r-1$. Thus, $x+a_i-a_r< N+a_r-1$. Also, for any $1\leq i\leq r-1$, $x+a_i-a_r>x+a_1-a_r$ since $a_i> a_1$. Now, $x+a_i-a_r> x+a_1-a_r\geq a_r+1+a_1-a_r=a_1+1$ since $0<a_r+1<a_r+2<\cdots < N+a_r-1$. It follows that $x+a_i-a_r> a_1+1$.

Combining the above, we get $a_1+1<x+a_i-a_r<N+a_r-1$. This implies that $f(x+a_i-a_r)$ depends only on the following terms 
\begin{align*}
    &f(a_1+1),f(a_1+2),\ldots, f(0),\\
    &f(1), f(2), \ldots, f(a_r),\\
    &f(a_r+1),f(a_r+2),\ldots, f(N-1)\\
    &f(N), f(N+1), f(N+2), \ldots , f(N+a_r-1).
\end{align*}

We know $0=f(a_1+1)=f(a_1+2)=\cdots =f(0)$, $1=f(N)= f(N+1)=f(N+2)= \cdots =f(N+a_r-1)$ and $f(1)=d_1,f(2)=d_2,\ldots,f(j)=d_j$ where $j=a_r$. For any $a_r+1\leq x \leq N-1$,
$f(a_r+1)$ is a linear combination of $f(1)$ and $f(a_1+1), f(a_2)+1,\ldots, f(a_{r-1}+1)$. Recall that $a_1+1\leq a_2, a_2+1\leq a_3, \ldots, a_{r-1}+1\leq a_r$. Hence, $f(a_r+1)$ depends on at most the expression $f(a_r)$, which is known. Thus, $f(a_r+1)$ is a linear combination of the elements in $S$. Since $f(x)$ is a recursive formula, for any $x>a_r+1$, $f(x)$ will be a linear combination of the variables in $S$. Simultaneously, we have $f(N)=f(N+1)=\ldots =f(N+a_r-2)=f(N+a_r-1)=1$. Therefore, $f(N), f(N+1), f(N+2), \ldots , f(N+a_r-1)$ are linear combinations of $\{d_1,d_2,\ldots,d_j\}$ where $j=a_r$ and are all equal to $1$. Hence, solving this system of $a_r$ equations with $a_r$ unknowns yields solutions for the variables in $S$ which provide solutions for the rest of the expressions, namely $f(a_r+1),f(a_r+2),\ldots, f(N-1)$.

\begin{example}\label{ex:faster version N=5}(Continuing Example \ref{ex:slow version N=5})
Let $N=5$ and $P=[[-2,1/2],[1,1/4],[2,1/4]]$ as in \Cref{ex:slow version N=5}. Recall the boundary conditions: $f(-1)=f(0)=0$ and  $f(5)=f(6)=1$. Note that $a_3=2$, the maximum of $\{-2,1,2\}$, introduces the new variables $\{d_1,d_2\}$ such that $f(1)=d_1$. Using the recurrence relation  $f(x)=4f(x-2)-2f(x-4)-f(x-1)$, we construct the equations for $f(3), f(4) , f(5)$, and $f(6)$, with $r=3$:
\begin{align*}
    f(-1)&=f(0)=0 \\
    f(1)&=d_1\\
    f(2)&=d_2 \\
    f(3)&=4f(1)-2f(-1)-f(2)= 4d_1-d_2\\
    f(4)&=4f(2)-2f(0)-f(3)=5d_2-4d_1 \\
    f(5)&=4f(3)-2f(1)-f(4)= 18d_1-9d_1\\
    f(6)&=4f(4)-2f(2)-f(5)=27d_2-34d_1.
\end{align*}
Since $f(5)=f(6)=1$, the system of equations $1= 18d_1-9d_1$ and $1=27d_2-34d_1$ has the solution $d_1=\frac{1}{5}$ and $d_2=\frac{13}{45}$. Substituting these values gives  
\[ f(1)=\frac{1}{5}, f(2)=\frac{13}{45}, f(3)=\frac{23}{45}, f(4)=\frac{29}{45}.\]
These results agree with \Cref{ex:slow version N=5}.
\end{example}

\subsection{Expected duration}\label{subsec:exp}
In the previous subsection, we considered the probability for the particle to reach some position $\geq N$. In this section, we consider the expected duration for the particle to end at a position $\leq 0$ or a position $\geq N$. Since this is analogous to the probability case, we omit the details.

Define $g(x)$ as the expected number of steps that a particle starting at $x$ will eventually reach a position $\leq 0$ or a position $\geq N$. For $0<x<N$, this expected duration satisfies the recurrence relation
\begin{equation}\label{eq:generalexp}
    g(x)=\sum_{i=1}^r p_ig(x+a_i)+1,
\end{equation}
where $a_i\in \ZZ$, $a_1<a_2<\cdots <a_r$, $a_1<0$, $a_r>0$ and boundary conditions $0=g(a_1+1)= g(a_1+2)=\cdots=g(-1)=g(0)$ and $0=g(N)=g(N+1)=\cdots =g(N+a_r-2)=g(N+a_r-1)$.

Similar to the probability case, we rewrite \Cref{eq:generalexp} to 
\begin{equation}\label{eq:bettergeneralexp}
g(x)=\frac{1}{p_r}g(x-a_r)-\frac{1}{p_r}\sum_{i=1}^{r-1}p_ig(x+a_i-a_r)-\frac{1}{p_r}
\end{equation}
where $0=g(a_1+1)= g(a_1+2)=\cdots=g(-1)=g(0)$ and $0=g(N)=g(N+1)=\cdots =g(N+a_r-2)=g(N+a_r-1)$.

Using \Cref{eq:bettergeneralexp} reduces the computational running time for the expected duration of the game as we will see in the next subsection.

\subsection{Comparison between the slower method and the faster method}\label{subsec:time comparison in 1d}
In this subsection, we compare the computational running times in Maple between the classical approach to solving $N-1$ linear equations and the faster method introduced in the previous section. Specically, we evaluate the performance of \lstinline{GR1dLG} and \lstinline{NewGR1dLG}, as described in \Cref{appendix:ggr1d}, Using the commands \lstinline{time(GR1dLG)} and \lstinline{time(NewGR1dLG)}, we measure the execution time for varying values of $N$ with $P=[[-1,\frac{1}{3}],[1,\frac{1}{3}],[2,\frac{1}{3}]]$. The measure timed in Maple are summarized below:
\begin{table}[H]
\centering
\begin{tabular}{|c|c|c|c|}
\hline
$N$ & Slow Method & Faster Method\\
&(seconds) & (seconds)\\
\hline
$100$ & $0.796$ & $0.015$ \\ \hline
$105$ & $11.171$ & $0.015$ \\ \hline
$110$ & $21.093$ & $0.015$ \\ \hline
$115$ & $500.828$ & $0.015$ \\ \hline
$120$ & $2466.609$ & $0.046$ \\
\hline
\end{tabular}
\end{table}

The data in the table suggests that introducing symbolic variables in \Cref{eq:bettergeneralexp} dramatically reduces the running time for computing the expected durations of the game starting at all starting locations.

\subsection{Variance} In this subsection, we consider the variance of the duration of the generalized $1$-dimensional gambler's ruin game. The computation builds up the expected duration discussed in \Cref{subsec:exp}. 

Let $P=[[a_1,p_1],[a_2,p_2],\ldots,[a_r,p_r]]$ such that $p_1+\ldots+p_r=1$ and define $F(x,t)$ as the probability generating function of the duration of the generalized $1$-dimensional gambler's ruin game. For $0<x<N$, this functions satisfies the recurrence relation

\begin{equation*}
F(x,t)=t(p_1F(x+a_1,t)+p_2F(x+a_2,t)+\cdots+p_rF(x+a_r,t)),
\end{equation*}

where $F(0,t)=1$ and $F(N,t)=1$. Making the substitution $t\mapsto z+1$ yields

\begin{equation}\label{eq:1d gen fun in z}
F(x,z)=(z+1)(p_1F(x+a_1,z)+p_2F(x+a_2,z)+\cdots+p_rF(x+a_r,z),
\end{equation}

with $F(0,z)=1$ and $F(N,z)=1$. We derive an expression to estimate the second factorial moment. Expanding $F(x,z)$ as a Taylor series gives

\begin{equation}\label{eq:taylor in 1d}
    F(x,z)=1+g(x)z+\frac{h(x)}{2!}z^2+\cdots
\end{equation}

where $g(x)$ is the expected duration as defined as in \Cref{subsec:exp}, and $h(x)$ represents the second factorial moment at $x$. Substituting \Cref{eq:taylor in 1d} into \Cref{eq:1d gen fun in z} results in

\begin{align*}
    1+g(x)z+\frac{h(x)}{2}z^2 + \dots &=(1+z)\left( p_1\left(1+g(x+a_1)z+\frac{h(x+a_1)}{2}z^2 +\cdots \right)\right)\\
    &+(1+z)\left(p_2\left(1+g(x+a_2)z+\frac{h(x+a_2)}{2}z^2+\cdots\right)\right)\\
    &\quad \vdots\\
    &+(1+z)\left(p_r\left(1+g(x+a_r)z+\frac{h(x+a_r)}{2}z^2+\cdots\right)\right).
\end{align*}

\noindent From the previous expression, we extract the coefficient of $z^2$ to get an expression for $h(x)$. Hence,

\begin{equation*}
    h(x)-(p_1h(x+a_1)+p_2h(x+a_2)+\cdots + p_rh(x+a_r))=2(p_1g(x+a_1)+p_2g(x+a_2)+\cdots+p_rg(x+a_r))
\end{equation*}
where $0=g(a_1+1)= g(a_1+2)=\cdots=g(-1)=g(0)$ and $0=g(N)=g(N+1)=\cdots =g(N+a_r-2)=g(N+a_r-1)$. The sum $g(x)+h(x)$ gives the second moment for $0<x<N$. 

\noindent The variance at $x$, denoted $V(x)$, is computed as
\[ V(x)=g(x)+h(x)-(g(x))^2\]
where $g(x)$ and $h(x)$ are defined as follows

\begin{align*}
    &g(x)=\frac{1}{p_r}g(x-a_r)-\frac{1}{p_r}\sum_{i=1}^{r-1}p_ig(x+a_i-a_r)-\frac{1}{p_r}
    \intertext{and}
    &h(x)-\sum_{i=1}^{r}p_ih(x+a_i)=2\sum_{i=1}^rp_ig(x+a_i).
\end{align*} 

We conclude this section with a table comparing the expected duration and the standard deviation when the particle starts at $x=\frac{N}{2}$ for various $N$ and probability table $P=[[-2,\frac{1}{3}],[1,\frac{1}{3}],[2,\frac{1}{3}]]$ .

\begin{table}[H]
\centering
\begin{tabular}{|c|c|c|c|}
\hline
$N$ & $x$ & Expected Duration& Standard Deviation \\
\hline
$10$ & $5$ & $ 8.613479400$ & $ 6.321808669$\\\hline
$20$ & $10$ & $  25.23344696$ & $18.44137538$\\ \hline
$30$ & $15$ & $42.94261730$ & $ 29.22243692$\\ \hline
$40$ & $20$ & $59.58246747$ & $ 37.26482832$ \\ \hline
$50$ & $25$ & $75.36543964$ & $ 43.30155080$\\ \hline
$60$ & $30$ & $90.70157954$ & $48.13687964$\\ \hline
$70$ & $35$ & $105.8379590$ & $52.27336865 $\\ \hline
$80$ & $40$ & $120.8913756$ & $55.98253147$\\ \hline
$90$ & $45$ & $135.9117966$ & $59.40593597$\\ \hline
$100$ & $50$ & $150.9194653$ & $62.61931702$ \\ \hline
\end{tabular}
\end{table}

\section{$2$-Dimensional Gambler's Ruin}\label{sec:2d case}

In this section, we consider the $2$-dimensional gambler's ruin and provide analogous results to the probability, expected duration and variance discussed in \Cref{sec:general 1d}. We introduce symbolic variables to the recurrence equation of the probability and expected duration. We then compare the computational running times in Maple between our method and the formula provided by Kmet and Petkov\v{s}ek. We conclude this section with an analysis of the variance.

We begin by recalling the setup for the $2$-dimensional case as described in \Cref{subsec:2 dim}. Consider a particle starting at a point $(x,y)$ in the interior of a rectangular grid of size $M\times N$, where $0<x<M$ and $0<y<N$. At each step, the particle moves in one of four directions, with probabilities $p_L, p_U, p_R,$ and $p_B$, corresponding to left, up, right and down movements, respectively. The particle stops moving once it hits one of the four boundaries, defined by $x=0, x=M, y=0$, or $y=N$.

\subsection{Probability}
We begin by considering the recurrence relation for the probability that the particle reaches some position at one of the four boundaries. We then rewrite this recurrence and introduce new variables for the probabilities at each $x$ and $y$.

Define $f(x,y)=f_L(x,y)L+f_R(x,y)R+f_U(x,y)U+f_D(x,y)D$, where $f_L(x,y)$ is the probability that the particle, starting at $(x,y)$ will exit the rectangle on the left side,
and analogously for $f_R(x,y),f_U(x,y),f_D(x,y)$, and $L,R,U,D$ are formal variables.
For $0<x<M$ and $0<y<N$, this probability satisfies the recurrence relation
\begin{equation}\label{eq:twodim prob}
    f(x,y)=p_Lf(x-1,y)+p_Uf(x,y+1)+p_Rf(x+1,y)+p_Bf(x,y-1)
\end{equation}

with boundary conditions $f(0,y)=L, f(x,N)=U, f(M,y)=R$ and $f(x,0)=B$ for all $0<x<M$ and $0<y<N$. 
Thus, the coefficients of $L, U, R,$ and $B$ in $f(x,y)$ represent the respective probabilities that a particle starting at $(x,y)$ will end on each boundary. For any $0<x<M$ and $0<y<N$, $f(x,y)$ can be determined by solving a system of $(M-1)\times(N-1)$ linear equations with $(M-1)\times(N-1)$ unknowns.

In the next subsection, we will introduce symbolic variables into the recurrence relation for the expected duration, as outlined in \Cref{sec:general 1d}. This method, applied to Equation \ref{eq:twodim prob} , 
but the same approach works for efficient computations of probabilities.
This is further detailed in \Cref{appendix:ggr2d}.

\subsection{Expected duration}\label{subsec:exp in 2d} We now consider the expected duration until the particle reaches some position at one of the four boundaries. We focus on the case when $p_W=p_N=p_E=p_S=\frac{1}{4}$, and remark that one can adapt our strategy for $p_W\neq p_N\neq p_E\neq p_S$ as described in \Cref{appendix:ggr2d}.

Define $g(x,y)$ as the expected number of steps that a particle starting at $(x,y)$ will eventually take to reach a position at one of the boundaries. For $0<x<M$ and $0<y<N$, this expected duration satisfies the recurrence relation
\begin{equation}\label{eq:2d expected duration}
    g(x,y)=\frac{1}{4}g(x-1,y)+\frac{1}{4}g(x,y+1)+\frac{1}{4}g(x+1,y)+\frac{1}{4}g(x,y-1)+1
\end{equation}

with boundary conditions $g(0,y)=g(M,y)=g(x,0)=g(x,N)=0$ for all $0<x<M$ and $0<y<N$. Similarly to the probability case, $g(x,y)$ can be determined  by solving a system of $(M-1)\times(N-1)$ linear equations with $(M-1)\times(N-1)$ unknowns.

Orr and Zeilberger provided a solution that reduces the number of linear equations from $(M-1)\times(N-1)$ to $\mathcal{O}(N+M)$. Their approach exploits symmetry by solving $g(0,y)=g(M-1,y)$ for $0<y<N$, and $g(x,1)=g(x,N-1)$ for $0<x<M$ \cite{OrrDoron}. Kmet and Petkov\v{s}ek gave an explicit solution involving a double sum, enabling direct computation of the expected duration without the need to solve systems of equations or use recursion \cite{Kmet}. For the special case where $M=N$, they established the following result.

\begin{theorem}[\cite{Kmet}, Equation (11)]
Consider the 2-dimensional gambler's ruin problem as stated in this section. Then, the expected duration of the game when the particle starts at $(x,y)$ is given by
\begin{equation*}
    g(x,y)=\frac{4}{M^2}\sum_{\substack{k=1 \\ k \text{ odd}}}^{M-1} \left(\sin{\left(\frac{jk\pi}{M}\right)}\cot{\left(\frac{k\pi}{2M}\right)} \sum_{\substack{\ell=1 \\ \ell \text{ odd}}}^{M-1} \frac{\sin{(\frac{i\ell\pi}{M})}\cot{(\frac{\ell\pi}{2M})}}{\sin^2{(\frac{k\pi}{2M})}+\sin^2{(\frac{\ell \pi}{2M})}}\right)
\end{equation*}
for any $0\leq x,y\leq M$.
\end{theorem}

We propose an alternative approach to those of Orr and Zeilberger, and Kmet and Petkov\v{s}ek. By rewriting \Cref{eq:2d expected duration}, we obtain the recurrence
\begin{equation}\label{eq:2d new expected recurrence}
    g(x,y)=4g(x-1,y)-g(x-1,y-1)-g(x-1,y+1)-g(x-2,y)-4
\end{equation}

where $x\to x-1$. The boundary conditions remain $g(0,y)=g(M,y)=g(x,0)=g(x,N)=0$ for all $x,y$. For each $1\leq j \leq N-1$ define $d_j=g(1,j)$ and construct the set $S=\{d_1,d_2,\ldots, d_{N-1}\}$. Using these variables, we reduce the system of equations. Initially, $g(x,y)$ is solved by constructing a system of $(M-1)\times(N-1)$ linear equations with $(M-1)\times(N-1)$ unknowns. By using $S$,  we reformulate the system into:

\begin{itemize}[leftmargin=0.2 in]
    \item $N-1$ equations for $x=1$,
    \item $(M-2)\times(N-1)$ equations for $2\leq x\leq M-1$, and
    \item $N-1$ boundary equations for $x=M$.
\end{itemize}

The boundary equations, derived using \Cref{eq:2d new expected recurrence}, are expressed in terms of the variables in $S$ and reduce the system to $N-1$ boundary equations with $N-1$ unknowns. Once these $N-1$ variables are solved, the remaining $(M-2)\times(N-1)$ equations can be determined. We illustrate this process in \Cref{ex:new approach}.

\begin{example}\label{ex:new approach}
Let $M=N=3$. This system consists of $4$ linear equations with $2$ unknowns: $d_1=g(1,1)$ and $d_2=g(1,2)$. The boundary conditions are $g(0,y)=g(M,y)=g(x,0)=g(x,N)=0$ for all $0\leq x,y\leq 3$. Using the recurrence in \Cref{eq:2d new expected recurrence}, we derive
\begin{align*}
    d_1&=g(1,1)\\
    d_2&=g(1,2)\\
    g(2,1)&=4d_1-d_2-4\\
    g(2,2)&=4d_2-d_1-4\\
    g(3,1)&=0=16d_1-8d_2-16\\
    g(3,2)&=0=16d_2-8d_1-16.
\end{align*}
Solving the last two equations yields $d_1=2$ and $d_2=2$. Substituting these values into the earlier equations gives $g(1,1)=g(1,2)=g(2,1)=g(2,2)=2$.
\end{example}

\subsection{Comparison between Kmet and Petkov\v{s}ek formula and faster method}\label{subsec:time comparison in 2d}
We compare the computational running times in Maple between Kmet and Petkov\v{s}ek's formula for the expected duration and the faster method introduced in the previous subsection. Specifically, we evaluate the performance of \lstinline{NewGR2dL} and \lstinline{KmetPetkovsek},  as described in \Cref{appendix:ggr2d}. Using the commands \lstinline{time(evalf(NewGR2dL(M,M)))} and \lstinline{time(KmetPetkovsek(M)))}, we measure the execution time for varying values of $M$. The results provide a direct comparison of the efficiency of the two methods. The measured times in Maple are summarized below:

\begin{table}[H]
\centering
\begin{tabular}{|c|c|c|c|c|c|c|}
\hline
$M$ & Faster Method & Kmet and Petkov\v{s}ek & $M$ & Faster Method & Kmet and Petkov\v{s}ek \\
&(seconds) & (seconds)& &(seconds) & (seconds)\\
\hline
10 & 0.015 & 0.015  & 90 & 9.843 &  181.953 \\ \hline
20 & 0.015 &  0.500 & 100 & 14.765 & 301.906\\ \hline
30 & 0.093 & 1.859  & 110 & 21.171 & 451.625\\ \hline
40 & 0.343 & 5.703  & 120 & 36.187 &  641.609\\ \hline
50 & 0.468 & 15.171 & 130 & 57.281  & 872.015 \\ \hline
60 & 1.171 & 35.125 & 140 & 80.140  &  1155.281\\ \hline
70 & 2.640 & 63.062 & 150 & 133.109 &  1564.515\\ \hline
80 & 5.062 &  114.234 & 160 &  226.312& 2023.125\\ \hline
\end{tabular}
\end{table}

While Kmet and Petkov\v{s}ek's formula expresses the expected duration as a double sum, it is computationally less efficient compared to our method.

\subsection{Variance} In this subsection, we analyze the variance of the duration of the $2$-dimensional gambler's ruin game under the condition $p_W=p_N=p_E=p_S=\frac{1}{4}$. The computation builds upon the expected duration discussed \Cref{subsec:exp in 2d}.

Define $F(x,y,t)$ as the probability generating function of the duration of the $2$-dimensional gambler's ruin game. For $0<x<M$ and $0<y<N$, this function satisfies the recurrence relation

\begin{equation}\label{eq:2d variance eq in t}
    F(x,y,t)=\frac{t}{4}(F(x-1,y,t)+F(x+1,y,t)+F(x,y-1,t)+F(x,y+1,t))
\end{equation}

where $F(0,0,t)=1$ and $F(M,N,t)=1$. Making the substitution $t\mapsto z+1$ yields

\begin{equation}
    F(x,y,z)=\frac{1+z}{4}(F(x-1,y,z)+F(x+1,y,z)+F(x,y-1,z)+F(x,y+1,z)),
\end{equation}

with $F(0,0,z)=1$ and $F(M,N,z)=1$. We derive an expression to estimate the second factorial moment. Expanding $F(x,y,z)$ as a Taylor series gives

\begin{equation}\label{eq:taylor in 2d}
    F(x,y,z)=1+g(x,y)z+\frac{h(x,y)}{2!}z^2 +\cdots
\end{equation}

where $g(x,y)$ is the expected duration as defined in \Cref{subsec:exp in 2d}, and $h(x,y)$ represents the second factorial moment at $(x,y)$. Substituting \Cref{eq:taylor in 2d} into \Cref{eq:2d variance eq in t} and extracting the coefficient of $z^2$ yields an expression for $h(x,y)$:

\[4h(x,y)-h(x-1,y)-h(x+1,y)-h(x,y-1)-h(x,y+1)=2(g(x-1,y)+g(x+1,y)+g(x,y-1)+g(x,y+1))\]
where $g(0,y)=g(M,y)=g(x,0)=g(x,N)=0$ for all $x,y$. The sum $g(x,y)+h(x,y)$ gives the second moment for $0<x<M$ and $0<y<N$. 

The variance at $(x,y)$, denoted by $V(x,y)$, is computed as
\[ V(x,y)=g(x,y)+h(x,y)-(g(x,y))^2\]
where $g(x,y)$ and $h(x,y)$ are defined by
\[g(x,y)=4g(x-1,y)-g(x-1,y-1)-g(x-1,y+1)-g(x-2,y)-4\]
and
\[
    4h(x,y)-h(x-1,y)-h(x+1,y)-h(x,y-1)-h(x,y+1)=2(g(x-1,y)+g(x+1,y)+g(x,y-1)+g(x,y+1)).\]

We conclude this section with a table comparing the expected duration and the standard deviation when the particle starts at $(x,y)=(\frac{M}{2},\frac{M}{2})$ for various $M$ under probabilities $P=[\frac{1}{6},\frac{1}{3},\frac{1}{6},\frac{1}{3}]$.

\begin{table}[H]
\centering
\begin{tabular}{|c|c|c|c|}
\hline
$M$ & $(x,y)$ & Expected Duration& Standard Deviation \\
\hline
$2$ & $(1,1)$ & $1$ & $0$\\\hline
$4$ & $(2,2)$ & $4.470588235$ & $ 2.891342524$\\ \hline
$6$ & $(3,3)$ & $10.3030$ & $7.102295958$\\ \hline
$8$ & $(4,4)$ & $18.47746573 $ & $12.97689858$ \\ \hline
$10$ & $(5,5)$ & $28.99020033 $ & $20.52455308$\\ \hline
$12$ & $(6,6)$ & $ 41.84019550$ & $29.74741677$\\ \hline
$14$ & $(7,7)$ & $ 57.02707373$ & $40.64621816$\\ \hline
$16$ & $(8,8)$ & $ 74.55067222$ & $53.22126050$\\ \hline
$18$ & $(9,9)$ & $94.41091165 $ & $67.47268859$\\ \hline
$20$ & $(10,10)$ & $116.6077497 $ & $83.40057864$ \\ \hline
\end{tabular}
\end{table}

\section{A Mirror Step Variant of Gambler's Ruin}\label{sec:mirror step}

In this section, we begin by formulating a new generalization of the gambler's ruin problem in $1$-dimension. A particle starts at some point $x$ on a line of length $N$ where $0< x <N$. At each step, the particle moves from $x$ to $x-1$ with probability $q_1$, or moves from $x$ to $x+1$ with probability $q_2$, or moves from $x$ to $N-x$ with probability $p$ where $0<p<1$ and $q_1+q_2+p=1$. We call this last step the \textit{mirror step}. The particle continues to walk on the line until it reaches $0$ or $N$. We focus on the case when $q_1=q_2=\frac{1-p}{2}$ and we call this the \textit{symmetric case}.

\subsection{Probability}
Define $f(x)=f_N^{(p)}(x)$ as the probability that a particle starting at $x$ will eventually reach $N$. For $0<x<N$, this probability satisfies the recurrence relation
\begin{equation} \label{eq:prob recurrence mirror}
    f(x)=\frac{1-p}{2}f(x-1)+\frac{1-p}{2}f(x+1)+pf(N-x)
\end{equation}
where $f(0)=0$ and $f(N)=1$. Before providing the solution to \Cref{eq:prob recurrence mirror}, we will try to guess the limit
as $N$ goes to infinity of the probability when the particle starts at some $x$ and ends at $N$, by using a fixed, large,  $N$ as discussed in \Cref{subsec:mirror}. We are interested in

\begin{equation}\label{eq:limit of fN(x)}
    \lim_{N\to \infty} f_N^{(p)}(x)
\end{equation}

and we hope to get expressions for when $x=2$ and $x=N-2$. First, we describe the approach that will generate data for fixed values of $N$ and $x$. Using the data, we can make some guesses for the limit in \Cref{eq:limit of fN(x)}.

\begin{itemize}[leftmargin=0.2 in]
    \item Fix $N$ as large as possible and use Maple to generate $N-1$ linear equations with \Cref{eq:prob recurrence mirror}.
    \item Solve for the $N-1$ linear equations.
    \item Generate data for $f_N^{(p)}(x)$ for different $p$ values with fixed $N$ and a fixed location $x$. 
    \item We use the function \texttt{identify} in Maple to try to identify the numerical values given by $f_N^{(p)}(x)$. The function \texttt{identify} is based, in part, on the continued fraction expansion of the numerical values. As $N$ grows the numerical value will converge to some number.
    \item Guess a formula for the number with fixed $x$ and varying $p$.
\end{itemize}

\noindent We begin with an example on how to generate data in Maple using the code accompanying this article.

\begin{example}\label{ex: x=1 prob}
Let $N=100$ and $x=2$. We generate data for the probability that if the particle starts at $x=2$, the particle eventually ends at $100$. Let $p\in \{\frac{1}{2},\frac{1}{3},\frac{1}{4},\ldots, \frac{1}{10}\}$. 

We use the procedure \texttt{Lk(p,x,N)}, as described in \Cref{appendix:mirror}, in Maple which generates the following data in about $8.390$ seconds:
\begin{align*}
    T \coloneqq &[0.48528137423857029281, 0.46410161513775458705, 0.444, 0.42705098312484227231,\\ &0.41171425595857973499, 0.39811169380648470689, 0.38595282133533513790, 0.375, \\
    &0.36506306819388080622].
\end{align*}

\noindent The sequence $T$ reads as follows. If the particle starts at $x=2$ and $p=\frac{1}{2}$, the particle moves from $x$ to $x-1$ with probability $\frac{1-p}{2}=\frac{1}{4}$, or moves from $x$ to $x+1$ with probability $\frac{1-p}{2}=\frac{1}{4}$, or moves from $x$ to $100-x$ with probability $p=\frac{1}{2}$. Then, the probability of the particle starting at $x=2$ and ending at $100$ is $T_1=0.48528137423857029281$. Similarly, if the particle starts at $x=2$ and $p=\frac{1}{3}$, the particle moves from $x$ to $x-1$ with probability $\frac{1-p}{2}=\frac{1}{3}$, or moves from $x$ to $x+1$ with probability $\frac{1-p}{2}=\frac{1}{3}$, or moves from $x$ to $100-x$ with probability $p=\frac{1}{3}$. Then, the probability of the particle starting at $x=2$ and ending at $100$ is $T_2=0.46410161513775458705$. Thus, $T_i$ is the probability of the particle starting at $x=2$ and ending at $100$ for $p=\frac{1}{i+1}$ where $1\leq i\leq 9$.

Using \texttt{identify} in Maple for the sequence $T$, we conjecture that each of the probabilities in $T$ converge to  
\begin{align*}
    M\coloneqq&\left[-8+6 \sqrt{2}, -3+2 \sqrt{3}, \frac{4}{9}, 
\frac{-5+3 \sqrt{5}}{4}, 
\frac{-24 + 14\sqrt{6}}{25}, 
\frac{-7+4 \sqrt{7}}{9}, 
\frac{-32+36 \sqrt{2}}{49}, {\frac{3}{8}}, 
\frac{-40+22 \sqrt{10}}{81}\right].
\end{align*}
That is the probability of the particle starting at $x=2$ and ending at $N$ as $N$ grows converges to $M_i$ for $p=\frac{1}{i+1}$ where $1\leq i\leq 9$.
\end{example}

The previous example illustrates that when the particle starts at $x=2$, $\{f_N^{(p)}(2)\}$ converges fast. We state the following guess for the expression of the limit when $x=2$ and in \Cref{cor:probability at infinity} we provide a proof.

\begin{guess}\label{guess:x=2}
 If the particle starts at $x=2$, then
\[\lim_{N\to\infty}f_N^{(p)}(2)=\frac{2\sqrt{p}(1+p-2\sqrt{p})}{(1-p)^2}.\]
\end{guess}

Using the same approach from above, we can obtain data for $x=N-2$. We guess the following expression of the limit when $x=N-2$ and provide a proof in \Cref{cor:probability at infinity}.

\begin{guess}\label{guess:x=N-2}
 If the particle starts at $x=N-2$, then
\[\lim_{N\to\infty}f_N^{(p)}(N-2)=\frac{(1+p)(1+p-2\sqrt{p})}{(1-p)^2}.\]
\end{guess}

We were able to guess more expressions for $\lim_{N\to\infty} f_N(x)$ with other $x$ values. After we made these guesses, we established the following key lemma which provides a relation between $f(x)$ and $f(N-x)$ for any $0\leq x\leq N$. We stress that this is only true for the symmetric case when the probability of the particle moving from $x$ to $x-1$ is the same as the probability of the particle moving from $x$ to $x+1$.

\begin{lemma}\label{lem:mirror probability equal to 1}
Consider the symmetric case when $f(x)=\frac{1-p}{2}f(x-1)+\frac{1-p}{2}f(x+1)+pf(N-x)$ with boundary conditions $f(0)=0, f(N)=1$ for some $0<p<1$. For any $0\leq x\leq N$, the following identity holds \[f(x) + f(N-x)=1.\]
\end{lemma}

\begin{proof}
    Call $g(x)=1-f(N-x)$. We will show that $f(x)+f(N-x)=1$ by proving that $g(x)=f(x)$. Note that $g(0)=0$ and $g(N)=1$, and $f(x)=1-g(N-x)$. Using the recurrence in \Cref{eq:prob recurrence mirror}, we substitute $1-g(N-x)$ for $f(x)$ and obtain:
    
    \begin{equation*}
        g(N-x)=\frac{1-p}{2}g(N-(x-1))+\frac{1-p}{2}g(N-(x+1))+pg(x).
    \end{equation*}
    That is,
    \begin{equation*}
        g(x)=\frac{1-p}{2}g(x+1))+\frac{1-p}{2}g(x-1)+pg(N-x)
    \end{equation*}
    which has boundary conditions $g(0)=0$ and $g(N)=1$. Therefore, $g(x)=f(x)$ as desired.
\end{proof}

Lemma \ref{lem:mirror probability equal to 1} establishes that the sums of the probabilities when the particle starts at $x$ and when the particle starts at $N-x$ equals to 1. Using this identity, we rewrite \Cref{eq:prob recurrence mirror} to

\begin{equation}\label{eq:inhomogeneous eq}
    f(x)=\frac{p}{1+p}+\frac{1}{2}\left(\frac{1-p}{1+p}\right)f(x-1)+\frac{1}{2}\left(\frac{1-p}{1+p}\right)f(x+1),  \quad f(0)=0, f(N)=1.
\end{equation}

We can now derive the probability of the particle ending at $N$ if it starts at some $x$, for general $N, x$ and $p$.

\begin{theorem}\label{thm:probability mirror}
Consider the generalization of the gambler's ruin problem when we add a mirror step. Then, the probability of ending at $N$ starting at $x$ is given by
\begin{equation*}
    f(x)=\frac{1}{2}\frac{\left(\frac{1-\sqrt{p}}{1+\sqrt{p}}\right)^N+1}{\left(\frac{1+\sqrt{p}}{1-\sqrt{p}}\right)^N-\left(\frac{1-\sqrt{p}}{1+\sqrt{p}}\right)^N}\left(\frac{1+\sqrt{p}}{1-\sqrt{p}}\right)^x
    +\frac{1}{2}\frac{\left(\frac{1+\sqrt{p}}{1-\sqrt{p}}\right)^N+1}{\left(\frac{1-\sqrt{p}}{1+\sqrt{p}}\right)^N-\left(\frac{1+\sqrt{p}}{1-\sqrt{p}}\right)^N}\left(\frac{1-\sqrt{p}}{1+\sqrt{p}}\right)^x+\frac{1}{2}
\end{equation*}
whenever we restrict the particle moves by either moving from $x$ to $x-1$ with probability $q_1$, or from $x$ to $x+1$ with probability $q_2$, or from $x$ to $N-x$ with probability $p$ where $q_1=q_2=\frac{1-p}{2}$.
\end{theorem}

\begin{proof}
    We can solve for \Cref{eq:inhomogeneous eq} because it is an inhomogeneous recurrence relation. Hence, we find a homogeneous and an inhomogeneous solution. 
    The general solution to the homogeneous equation $f(x)=\frac{1}{2}\left(\frac{1-p}{1+p}\right)f(x-1)+\frac{1}{2}\left(\frac{1-p}{1+p}\right)f(x+1)$ is
    
    \begin{equation*}
        f(x)=A\left(\frac{1+p+2\sqrt{p}}{1-p}\right)^x+B\left(\frac{1+p-2\sqrt{p}}{1-p}\right)^x
    \end{equation*}
    
    for some numbers $A$ and $B$. 
    Next, we find the particular solution to the inhomogeneous relation by setting $f^*(x)=C$ for some constant $C$. Then,
    
    \begin{equation*}
        f^*(x)=\frac{p}{1+p}+\frac{1}{2}\left(\frac{1-p}{1+p}\right)f^*(x-1)+\frac{1}{2}\left(\frac{1-p}{1+p}\right)f^*(x+1)
    \end{equation*}
    
    becomes
    
    \begin{equation*}
        C=\frac{p}{1+p}+\frac{1}{2}\left(\frac{1-p}{1+p}\right)C+\frac{1}{2}\left(\frac{1-p}{1+p}\right)C
    \end{equation*}
    
    which has solution $C=\frac{1}{2}$. 
    
    Therefore, $f^*(x)=\frac{1}{2}$ is the particular solution, and the general inhomogeneous solution is
    
    \begin{equation}\label{eq:general inhomogeneous sln}
         f(x)=A\left(\frac{1+p+2\sqrt{p}}{1-p}\right)^x+B\left(\frac{1+p-2\sqrt{p}}{1-p}\right)^x+\frac{1}{2}.
    \end{equation}
    
    Using Maple, we find $A$ and $B$ by using the boundary conditions to get a system of two linear equations. Namely,
    
    \begin{align*}
        0&=A+B+\frac{1}{2} \intertext{and}
        \frac{1}{2}&=A\left(\frac{1+p+2\sqrt{p}}{1-p}\right)^N+B\left(\frac{1+p-2\sqrt{p}}{1-p}\right)^N.
    \end{align*}
    
   Solving for the above linear equations, we get
   \begin{align}
       A&=\frac{1}{2}\frac{\left(\frac{1+p-2\sqrt{p}}{1-p}\right)^N+1}{\left(\frac{1+p+2\sqrt{p}}{1-p}\right)^N-\left(\frac{1+p-2\sqrt{p}}{1-p}\right)^N} \label{eq: A eq}\intertext{and}
       B&=\frac{1}{2}\frac{\left(\frac{1+p+2\sqrt{p}}{1-p}\right)^N+1}{\left(\frac{1+p-2\sqrt{p}}{1-p}\right)^N-\left(\frac{1+p+2\sqrt{p}}{1-p}\right)^N}.\label{eq: B eq}
   \end{align}
   
   Rewriting Equations \ref{eq:general inhomogeneous sln}, \ref{eq: A eq} and \ref{eq: B eq} gives
   \begin{align*}
       f(x)&=A\left(\frac{1+\sqrt{p}}{1-\sqrt{p}}\right)^x+B\left(\frac{1-\sqrt{p}}{1+\sqrt{p}}\right)^x+\frac{1}{2} 
   \end{align*}
as desired.
\end{proof}

\Cref{thm:probability mirror} provides a formula for $\lim_{N\to\infty}f_N(x)$ for any $x$ where $0<x<N$.

\begin{corollary}\label{cor:probability at infinity}
If the particle starts at some $x$ where $0<x<N$, then
\[\lim_{N\to \infty} f_N^{(p)}(x)=\frac{1}{2}-\frac{1}{2}\left(\frac{1-\sqrt{p}}{1+\sqrt{p}}\right)^x\]
whenever we restrict the particle moves by either moving from $x$ to $x-1$ with probability $q_1$, or from $x$ to $x+1$ with probability $q_2$, or from $x$ to $N-x$ with probability $p$ where $q_1=q_2=\frac{1-p}{2}$.
\end{corollary}

\begin{proof}
Let 

\begin{equation}\label{eq: fNx}
    f_N^{(p)}(x)=\frac{1}{2}\frac{\left(\frac{1-\sqrt{p}}{1+\sqrt{p}}\right)^N+1}{\left(\frac{1+\sqrt{p}}{1-\sqrt{p}}\right)^N-\left(\frac{1-\sqrt{p}}{1+\sqrt{p}}\right)^N}\left(\frac{1+\sqrt{p}}{1-\sqrt{p}}\right)^x
    +\frac{1}{2}\frac{\left(\frac{1+\sqrt{p}}{1-\sqrt{p}}\right)^N+1}{\left(\frac{1-\sqrt{p}}{1+\sqrt{p}}\right)^N-\left(\frac{1+\sqrt{p}}{1-\sqrt{p}}\right)^N}\left(\frac{1-\sqrt{p}}{1+\sqrt{p}}\right)^x+\frac{1}{2},
\end{equation}

    and \[g_N^{(p)}(x)=\frac{1}{2}\frac{1}{\left(\frac{1+\sqrt{p}}{1-\sqrt{p}}\right)^N}\left(\frac{1+\sqrt{p}}{1-\sqrt{p}}\right)^x+\frac{1}{2}-\frac{1}{2}\left(\frac{1-\sqrt{p}}{1+\sqrt{p}}\right)^x.\]
    We will show that $\lim_{N\to\infty}f_N^{(p)}(x)-g_N^{(p)}(x)=0$.
    Observe that
    
    \begin{align*}
    f_N^{(p)}(x)-g_N^{(p)}(x)&=\frac{1}{2}\frac{\left(\frac{1-\sqrt{p}}{1+\sqrt{p}}\right)^N+1}{\left(\frac{1+\sqrt{p}}{1-\sqrt{p}}\right)^N-\left(\frac{1-\sqrt{p}}{1+\sqrt{p}}\right)^N}\left(\frac{1+\sqrt{p}}{1-\sqrt{p}}\right)^x
    +\frac{1}{2}\frac{\left(\frac{1+\sqrt{p}}{1-\sqrt{p}}\right)^N+1}{\left(\frac{1-\sqrt{p}}{1+\sqrt{p}}\right)^N-\left(\frac{1+\sqrt{p}}{1-\sqrt{p}}\right)^N}\left(\frac{1-\sqrt{p}}{1+\sqrt{p}}\right)^x\\
    & \quad -
    \frac{1}{2}\frac{1}{\left(\frac{1+\sqrt{p}}{1-\sqrt{p}}\right)^N}\left(\frac{1+\sqrt{p}}{1-\sqrt{p}}\right)^x+\frac{1}{2}\left(\frac{1-\sqrt{p}}{1+\sqrt{p}}\right)^x\\
    &=\frac{1}{2}\left(\frac{1+\sqrt{p}}{1-\sqrt{p}}\right)^x\left(\frac{\left(\frac{1-\sqrt{p}}{1+\sqrt{p}}\right)^N+1}{\left(\frac{1+\sqrt{p}}{1-\sqrt{p}}\right)^N-\left(\frac{1-\sqrt{p}}{1+\sqrt{p}}\right)^N}-\left(\frac{1-\sqrt{p}}{1+\sqrt{p}}\right)^N\right)\\
    &\quad+\frac{1}{2}\left(\frac{1-\sqrt{p}}{1+\sqrt{p}}\right)^x\left(\frac{\left(\frac{1+\sqrt{p}}{1-\sqrt{p}}\right)^N+1}{\left(\frac{1-\sqrt{p}}{1+\sqrt{p}}\right)^N-\left(\frac{1+\sqrt{p}}{1-\sqrt{p}}\right)^N}+\frac{1}{2}\right)\\
    &=\frac{1}{2}\left(\frac{1+\sqrt{p}}{1-\sqrt{p}}\right)^x\left(\frac{\left(\frac{1-\sqrt{p}}{1+\sqrt{p}}\right)^N+\left(\frac{1-\sqrt{p}}{1+\sqrt{p}}\right)^{2N}}{\left(\frac{1+\sqrt{p}}{1-\sqrt{p}}\right)^N-\left(\frac{1-\sqrt{p}}{1+\sqrt{p}}\right)^N}\right)-\frac{1}{2}\left(\frac{1-\sqrt{p}}{1+\sqrt{p}}\right)^x\left(\frac{1+\left(\frac{1-\sqrt{p}}{1+\sqrt{p}}\right)^N}{\left(\frac{1+\sqrt{p}}{1-\sqrt{p}}\right)^N-\left(\frac{1-\sqrt{p}}{1+\sqrt{p}}\right)^N}\right)\\
    &=\frac{1}{2}\frac{\left(1+\left(\frac{1-\sqrt{p}}{1+\sqrt{p}}\right)^N\right)\left( \left(\frac{1+\sqrt{p}}{1-\sqrt{p}}\right)^x\left(\frac{1-\sqrt{p}}{1+\sqrt{p}}\right)^N - \left(\frac{1-\sqrt{p}}{1+\sqrt{p}}\right)^x\right)}{\left(\frac{1+\sqrt{p}}{1-\sqrt{p}}\right)^N-\left(\frac{1-\sqrt{p}}{1+\sqrt{p}}\right)^N}\\
    &=\frac{1}{2}\frac{\left(1+ \left(\frac{1-\sqrt{p}}{1+\sqrt{p}}\right)^{N}\right)\left(\left(\frac{1+\sqrt{p}}{1-\sqrt{p}}\right)^x-1\right)}{\left(\frac{1+\sqrt{p}}{1-\sqrt{p}}\right)^{2N}-1}.
    \end{align*}

    Since
    
    \begin{align*}
        \lim_{N\to \infty} \left(\frac{1-\sqrt{p}}{1+\sqrt{p}}\right)^{N}=0 \intertext{and}
        \lim_{N\to \infty}\frac{1}{2}\frac{\left(\left(\frac{1+\sqrt{p}}{1-\sqrt{p}}\right)^x-1\right)}{\left(\frac{1+\sqrt{p}}{1-\sqrt{p}}\right)^{2N}-1}=0,
    \end{align*}
It follows that \[\lim_{N\to\infty}f_N^{(p)}(x)-g_N^{(p)}(x)=0.\]

Thus,
\[\lim_{N\to \infty} f_N^{(p)}(x) =\frac{1}{2}-\frac{1}{2}\left(\frac{1-\sqrt{p}}{1+\sqrt{p}}\right)^x.\]
\end{proof}

Using \Cref{cor:probability at infinity} provides a proof for \Cref{guess:x=1}. See the following example.
\begin{example}
Setting $x=1$,
\[\lim_{N\to \infty} f_N^{(p)}(1)=\frac{1}{2}-\frac{1}{2}\left(\frac{1-\sqrt{p}}{1+\sqrt{p}}\right)=\frac{\sqrt{p}-p}{1-p}\] as expected.
\end{example}

\subsection{Expected duration}

We now consider the expected duration of the gambler's ruin problem with a mirror step. Define $g(x)$ as the expected number of steps that a particle starting at $x$ will eventually reach a position $0$ or $N$. For $0<x<N$, this expected duration satisfies the recurrence relation
    \begin{equation} \label{eq:mirror exp}
        g(x)=\frac{1-p}{2}g(x-1)+\frac{1-p}{2}g(x+1) + pg(N-x) +1, \quad g(0)=0, g(N)=0.
    \end{equation}
We use this recurrence relation to find a closed formula for the expected duration of the game.
\begin{theorem}\label{thm:expected duration mirror}
Consider the generalization of the gambler's ruin problem when we add a mirror step. Then, the expected duration of ending at $0$ or $N$ starting at $x$ is given by
\begin{equation*}
    g(x)=\frac{1}{1-p}x(N-x)
\end{equation*}
whenever we restrict the particle moves by either moving from $x$ to $x-1$ with probability $q_1$, or from $x$ to $x+1$ with probability $q_2$, or jumps to $N-x$ with probability $p$ where $q_1=q_2=\frac{1-p}{2}$.
\end{theorem}
Remark: When $p=0$, \Cref{thm:expected duration mirror} recovers the formula for the expected duration of the classical gambler's ruin game.

\begin{proof}
Let $h(x)=\frac{1}{1-p}x(N-x)$ and observe that $h(0)=0, h(N)=0$. We prove that $h(x)$ satisfies the same recurrence relation as $g(x)$. Applying the recurrence from \Cref{eq:mirror exp} to $h(x)$ and simplifying yields 
\begin{align*}
    &\frac{1-p}{2}h(x-1) +\frac{1-p}{2}h(x+1)+ph(N-x)+1\\
    &=  \frac{1-p}{2}\left(\frac{1}{1-p}(x-1)(N-x+1)\right)+\frac{1-p}{2}\left(\frac{1}{1-p}(x+1)(N-x-1)\right)+p\left(\frac{1}{1-p}x(N-x)\right)+1\\
    &=\frac{1}{1-p}x(N-x).
\end{align*}

Hence, $h(x)$ satisfies the following recurrence relation 
\begin{equation*}
    h(x)=\frac{1-p}{2}h(x-1) +\frac{1-p}{2}h(x+1)+ph(N-x)+1, \quad h(0)=0, h(N)=0.
\end{equation*}
Thus, we get that $h(x)=g(x)$ which completes the proof.
\end{proof}

\section{Future Work }\label{sec:future}  

In \Cref{sec:mirror step}, we consider a generalization of the gambler's ruin problem where the particle starts at some point $x$ on a line of length $N$ where $0< x <N$. At each step, the particle either moves to the left by one step with probability $q_1$, moves to the right by one step with probability $q_2$, or moves to $N-x$ with probability $p$ where $0<p<1$ and $q_1+q_2+p=1$. We focus on the case when $q_1=q_2=\frac{1-p}{2}$, and we provide formulas for the probability that the particle ends at $N$ and the expected number of steps to finish the game. Thus, it is an open problem to give formulas for general $q_1, q_2$ and $p$. 

Computational evidence suggests the following conjecture when the probability of moving from $x$ to $x-1$ is the same as the probability of moving from $x$ to $N-x$. 
\begin{conjecture}
Consider the generalization of the gambler's ruin problem when we add a mirror step. If the particle starts at $x=1$, then
\[\lim_{N\to \infty} f_N^{(p)}(1)= \frac{\sqrt{(p+1)(1-3p+4p^2)}-(1-2p)(p+1)}{2p(p+1)}\]
whenever we restrict the particle moves by either moving from $x$ to $x-1$ with probability $q_1$, or from $x$ to $x+1$ with probability $p$, or from $x$ to $N-x$ with probability $q_2$ where $q_1=q_2=\frac{1-p}{2}$.
\end{conjecture}

\section*{Acknowledgements}
The author thanks her advisor Dr. Doron Zeilberger for the introduction
to the problem and feedback on an earlier draft. The author was supported by the NSF Graduate Research Fellowship Program under Grant No. 2233066.

\bibliographystyle{plain}
\bibliography{bibliography.bib}

\appendix
\section{Computational Tools for Analyzing the Gambler's Ruin Problem}\label{appendix:maple}
\subsection{Method Descriptions}
There are four text files accompanying this article: \lstinline{GGR.txt, GGR1d.txt, GGR2d.txt,} and \lstinline{GGR1dMirror.txt} which can be found in the GitHub repository \\
\href{https://github.com/marti310/Gamblers-Ruin}{https://github.com/marti310/Gamblers-Ruin}.

In this section, we will describe the functionality of
some of the main procedures. These text files should be saved in the same directory. All procedures were written and tested for Maple 20.

\subsubsection{\lstinline{GGR.txt}}
The \lstinline{GGR.txt} file contains the following main procedures.

\begin{itemize}
\item \lstinline{ProbN(N)}

    Returns a list $L$ of length $N-1$.
    
    This function inputs a positive integer $N$ and computes the probability of ending at $N$ for every $1\leq x \leq N-1$.

    Example:
    
    \begin{lstlisting}
    read `GGR.txt`:
    N:=10;
    ProbN(N);
    \end{lstlisting}
    Output:
    
    \lstinline!{[1/10, 1/5, 3/10, 2/5, 1/2, 3/5, 7/10, 4/5, 9/10]}!
    
    \item \lstinline{ExpN(N)}
    
    Returns a list $L$ of length $N-1$.
    
    This function inputs a positive integer $N$ and computes the expected number of steps of ending at either $0$ or $N$ for every $1\leq x \leq N-1$.
\end{itemize}

\subsubsection{\lstinline{GGR1d.txt}}\label{appendix:ggr1d}
The \lstinline{GGR1d.txt} file contains procedures for the classical gambler's ruin game. In addition, it contains procedures for the $1$-dimensional case. We provide the main procedures.

\begin{itemize}
\item \lstinline{GR1dPG(N,P)}

    Returns a list $L$ of length $N-1$.
    
    This function inputs a positive integer $N$ and a probability table $P$ where $P=[[a_1,p_1],[a_2,p_2],...,[a_r,p_r]]$ and computes the probability of the particle ending at some position $\geq N$ for every $1\leq x \leq N-1$.
    
    Remark: This procedures uses the classical approach of solving for $N-1$ linear equations.

    Example:
    
    \begin{lstlisting}
    read `GGR1d.txt`:
    P:=GR1dPG(10,[[-1,1/2],[1,1/2]]);
    \end{lstlisting}
    Output:
    
    \lstinline!{[1/10, 1/5, 3/10, 2/5, 1/2, 3/5, 7/10, 4/5, 9/10]}!

    \item \lstinline{GR1dLG(N,P)}
    
    Returns a list $L$ of length $N-1$.
    
    This function inputs a positive integer $N$ and a probability table $P$ where $P=[[a_1,p_1],[a_2,p_2],...,[a_r,p_r]]$ and computes the expected number of steps for the particle to end at $\leq 0$ or $\geq N$ for every $1\leq x \leq N-1$.

    Remark: This procedures uses the classical approach of solving for $N-1$ linear equations.

    \item \lstinline{NewGR1dPG(N,P)}
    
    Returns a list $L$ of length $N-1$.
    
    This function inputs a positive integer $N$ and a probability table $P$ where $P=[[a_1,p_1],[a_2,p_2],...,[a_r,p_r]]$ and computes the probability for the particle to end at $\geq N$ for every $1\leq x \leq N-1$.

    Remark: This procedures uses the faster method.

    Example:
    
    \begin{lstlisting}
    read `GGR1d.txt`:
    P:=NewGR1dPG(10,[[-1,1/2],[1,1/2]]);
    \end{lstlisting}
    Output:
    
    \lstinline!{[1/10, 1/5, 3/10, 2/5, 1/2, 3/5, 7/10, 4/5, 9/10]}!
    
     \item \lstinline{NewGR1dLG(N,P)}
    
    Returns a list $L$ of length $N-1$.
    
    This function inputs a positive integer $N$ and a probability table $P$ where $P=[[a_1,p_1],[a_2,p_2],...,[a_r,p_r]]$ and computes the expected number of steps for the particle to end at $\leq 0$ or $\geq N$ for every $1\leq x \leq N-1$.

    Remark: This procedures uses the faster method.
\end{itemize}

\subsubsection{\lstinline{GGR2d.txt}}\label{appendix:ggr2d}
The \lstinline{GGR2d.txt} file contains procedures for the $2$-dimensional gambler's ruin game. We provide the main procedures.

\begin{itemize}
\item \lstinline{GR2dP(M,N,L,U,R,B)}

    Returns an $(M-1) \times (N-1)$ matrix whose entries are linear combinations of L,U,R,B.
    
    This function inputs positive integers $M, N$ and symbols L,U,R,B where L is the left edge, U is the top edges, R is the right edge and B is the bottom edge of the $M\times N$ rectangle, and computes the probability of the particle starting at some point $(a,b)$ and ending on L, U, R or B for every $1\leq a \leq M-1$ and $1\leq b \leq M-1$.

    Remark: This procedures uses the classical approach of solving for $(M-1)\times(N-1)$ linear equations.

    Example:
    
    \begin{lstlisting}
    read `GGR2d.txt`:
    GR2dP(3,3,L,U,R,B);
    \end{lstlisting}
    Output:
    
    \lstinline!{[[(3L)/8 + (3B)/8 + U/8 + R/8, (3L)/8 + B/8 + (3U)/8 + R/8], [L/8 + (3B)/8 + U/8 + (3R)/8, L/8 + B/8 + (3U)/8 + (3R)/8]]}!

    \item \lstinline{GR2dL(M,N)}
    
    Returns an $(M-1)\times (N-1)$ matrix $M$.
    
    This function inputs positive integers $M, N$ and computes the expected number of steps of the particle starting at some point $(a,b)$ and ending on L, U, R or B for every $1\leq a \leq M-1$ and $1\leq b \leq M-1$.

     Remark: This procedures uses the classical approach of solving for $(M-1)\times(N-1)$ linear equations.
    
    Example:
    
    \begin{lstlisting}
    read `GGR2d.txt`:
    GR2dL(3,3);
    \end{lstlisting}
    Output:
    
    \lstinline!{[[2, 2], [2, 2]]}!

    \item \lstinline{NewGR2dP(M,N,L,U,R,B)}

    Returns an $(M-1) \times (N-1)$ matrix whose entries are linear combinations of L,U,R,B.
    
    This function inputs positive integers $M, N$ and symbols L,U,R,B where L is the left edge, U is the top edges, R is the right edge and B is the bottom edge of the $M\times N$ rectangle, and computes the probability of the particle starting at some point $(a,b)$ and ending on L, U, R or B for every $1\leq a \leq M-1$ and $1\leq b \leq M-1$.

    Remark: This procedures uses the faster method. 
    
    Example:
    
    \begin{lstlisting}
    read `GGR2d.txt`:
    NewGR2dP(3,3,L,U,R,B);
    \end{lstlisting}
    Output:
    
    \lstinline!{[[(3L)/8 + (3B)/8 + U/8 + R/8, (3L)/8 + B/8 + (3U)/8 + R/8], [L/8 + (3B)/8 + U/8 + (3R)/8, L/8 + B/8 + (3U)/8 + (3R)/8]]}!

    \item \lstinline{NewGR2dL(M,N)}
    
    Returns an $(M-1)\times (N-1)$ matrix $M$.
    
    This function inputs positive integers $M, N$ and computes the expected number of steps of the particle starting at some point $(a,b)$ and ending on L, U, R or B for every $1\leq a \leq M-1$ and $1\leq b \leq M-1$.

     Remark: This procedures uses the faster method.
    
    Example:
    
    \begin{lstlisting}
    read `GGR2d.txt`:
    NewGR2dL(3,3);
    \end{lstlisting}
    Output:
    
    \lstinline!{[[2, 2], [2, 2]]}!

      \item \lstinline{NewGR2dPG(M,N,L,U,R,B,P)}

    Returns an $(M-1) \times (N-1)$ matrix whose entries are linear combinations of L,U,R,B.
    
    This function inputs positive integers $M, N$, symbols L,U,R,B where L is the left edge, U is the top edges, R is the right edge and B is the bottom edge of the $M\times N$ rectangle, and a probability table $P=[p_L,p_U,p_R,p_B]$ such that the particle moves left by one step with probability $p_L$, or moves up by one step with probability $p_U$, or moves right by one step with probability $p_R$, or moves down by one step with probability $p_B$, and computes the probability of the particle starting at some point $(a,b)$ and ending on L, U, R or B for every $1\leq a \leq M-1$ and $1\leq b \leq M-1$.

    Remark: This procedures uses the faster method. 
    
    Example:
    
    \begin{lstlisting}
    read `GGR2d.txt`:
    NewGR2dPG(3,3,L,U,R,B,[1/4,1/4,1/4,1/4]);
    \end{lstlisting}
    Output:
    
    \lstinline!{[[R/8 + (3L)/8 + U/8 + (3B)/8, R/8 + (3L)/8 + (3U)/8 + B/8], [(3R)/8 + L/8 + U/8 + (3B)/8, (3R)/8 + L/8 + (3U)/8 + B/8]]}!
    \item \lstinline{NewGR2dLG(M,N,P)}

    Returns an $(M-1) \times (N-1)$ matrix.
    
    This function inputs positive integers $M, N$, and a probability table $P=[p_L,p_U,p_R,p_B]$ such that the particle moves left by one step with probability $p_L$, or moves up by one step with probability $p_U$, or moves right by one step with probability $p_R$, or moves down by one step with probability $p_B$, and computes the expected number of steps of the particle starting at some point $(a,b)$ and ending on L, U, R or B for every $1\leq a \leq M-1$ and $1\leq b \leq M-1$.

    Remark: This procedures uses the faster method. 
    
    Example:
    
    \begin{lstlisting}
    read `GGR2d.txt`:
    NewGR2dLG(3,3,[1/4,1/4,1/4,1/4]);
    \end{lstlisting}
    Output:
    
    \lstinline!{[[2, 2], [2, 2]]}!

     \item \lstinline{KmetPetkovsek(N)}

    Returns an $(N-1) \times (N-1)$ matrix.
    
    This function inputs a positive integer $N$ and implements Kmet and Petkovsek's formula for the expected duration of the 2-dimensional gambler's ruin game for the $M=N$ case. 
    
    Example:
    
    \begin{lstlisting}
    read `GGR2d.txt`:
    KmetPetkovsek(3);
    \end{lstlisting}
    Output:
    
    \lstinline!{[[2, 2], [2, 2]]}!
\end{itemize}

\subsubsection{\lstinline{GGR1dMirror.txt}}\label{appendix:mirror}
The \lstinline{GGR1dMirror.txt} file contains the following main procedures for the new generalization of gambler's ruin where we add a third step.

\begin{itemize}
\item \lstinline{ProbN2(N,P)}

    Returns a list $L$ of length $N$.
    
    This function inputs a positive integer $N$ and a probability table $P$ where $P=[p_1,p_2,p_3]$ where $p_1+p_2+p_3=1$ and outputs a list $L$ of length $N$ such that $L[x]$ is the probability of the particle ending at $N$ when it starts at $x$ where the particle can move from $x$ to $x-1$ with probability $p_1$, or $x$ to $x+1$ with probability $p_2$, or $x$ to $N-x$ with probability $p_3$.

    Example:
    
    \begin{lstlisting}
    read `GGR1dMirror.txt`:
    ProbN2(5,[1/3,1/3,1/3]);
    \end{lstlisting}
    Output:
    
    \lstinline!{[7/19, 9/19, 10/19, 12/19, 1]}!

    \item \lstinline{ExpN2(N,P)}

    Returns a list $L$ of length $N$.
    
    This function inputs a positive integer $N$ and a probability table $P$ where $P=[p_1,p_2,p_3]$ where $p_1+p_2+p_3=1$ and outputs a list $L$ of length $N$ such that $L[x]$ is the expected number of steps that the particle takes to end at $N$ or $0$ when it starts at $x$ where the particle can move from $x$ to $x-1$ with probability $p_1$, or $x$ to $x+1$ with probability $p_2$, or $x$ to $N-x$ with probability $p_3$.

    Example:
    
    \begin{lstlisting}
    read `GGR1dMirror.txt`:
    ExpN2(5,[1/3,1/3,1/3]);
    \end{lstlisting}
    Output:
    
    \lstinline!{[6, 9, 9, 6, 0]}!

     \item \lstinline{Lk(p,x,N)}

    Returns a number.
    
    This function inputs a probability value $p, 0<p<1$, and positive integers $x,N$ where $0<x<N$ and outputs the exact probability of the particle ending at $N$ when it starts at location $x$ where with probability $(1-p)/2$ the particle moves to $x-1$, with probability $(1-p)/2$ the particle moves to $x+1$ and with probability $p$ the particle moves to $N-x$.

    Example:
    
    \begin{lstlisting}
    read `GGR1dMirror.txt`:
    Lk(1/3,4,100);
    \end{lstlisting}
    Output:
    
    \lstinline!{0.4974226119}!
    
\end{itemize}

\end{document}